\documentclass[a4paper,12pt,letterpaper,english]{amsart}
\usepackage[utf8]{inputenc}
\usepackage[T1]{fontenc}
\usepackage{amsmath,amsthm}
\usepackage{amsfonts,amssymb,enumerate}
\usepackage{url,paralist}
\usepackage{mathtools,xcolor}
\usepackage[colorlinks=true,urlcolor=blue,linkcolor=blue,citecolor=blue]{hyperref}
\usepackage{enumerate}
\usepackage{relsize} 
\usepackage{anysize}
\usepackage[arrow,curve,matrix,tips,2cell]{xy}
\usepackage{lscape}

\theoremstyle{plain}
\newtheorem{theorem}{Theorem}[section]
\newtheorem{lemma}[theorem]{Lemma}

\newtheorem*{theorem*}{Theorem}

\newtheorem*{claim*}{Claim}
\newtheorem*{lemma*}{Lemma}

\theoremstyle{definition}

\begin{document}
\baselineskip=15.5pt
\title[involutions fixing $F^n\cup F^4$]{Determination of bounds on the dimension of manifolds with involutions fixing $F^n\cup F^4$}
 
\author[A. Nath]{Arijit Nath} 
\address{Indian Institute of Science Education and Research, Berhampur\\
Transit campus, (Govt. ITI Building), Engg. School Junction, Berhampur, Odisha 760010, India}	
	
\email{arijit2357@gmail.com, arijitnath@iiserbpr.ac.in}

\author[A. Nath]{Avijit Nath} 
\address{Indian Institute of Science Education and Research, Berhampur\\ Transit campus, (Govt. ITI Building), Engg. School Junction, Berhampur, Odisha 760010, India}
\email{avijitnath.math@gmail.com, avijitnath@iiserbpr.ac.in}

\thispagestyle{empty}
\date{}

\begin{abstract}
Let $M^m$ be an $m$-dimensional, smooth and closed manifold, equipped with a smooth involution $T\colon M^m \to M^m$ fixing submanifolds $F^n$ and $F^4$ of dimensions $n$ and $4$, respectively, where $4<n<m$ and $F^n\cup F^4$ does not bound. We determine the upper bound for $m$, for each $n$. The existence of these bounds is ensured by the famous Five Halves Theorem of J. Boardman, which establishes that,
under the above hypotheses, $m\leqslant\frac{5}{2}n$.

\end{abstract}

\dedicatory{}

\maketitle

\section{Introduction} \label{intro}
Suppose $M^m$ is a smooth and closed $m$-dimensional manifold and $T\colon M^m \to M^m$ is a smooth involution defined on $M^m$. The fixed point set of $T$, $F$, is a disjoint union of closed submanifolds of $M^m$, $F=\bigcup\limits_{j=0}^{n} F^j$ where $F^j$ denotes the union of those components of $F$ having dimension $j$. It is well-known, from equivariant
bordism theory, that if $(M^m, T)$ is non-bounding then $n$ cannot be too small with respect to $m$. This fact was evidenced from an old result of Conner and Floyd (Theorem 27.1 of \cite{cf}), which stated: for each natural number $n$, there exists a number $\phi(n)$ with the property that, if $m>\phi(n)$, then $(M^m,T$) bounds equivariantly. Later, this was explicitly
confirmed by the famous $5/2$-Theorem of Boardman \cite{boardman}: if $M^m$ is nonbounding, then $m\leqslant\frac{5}{2}n$. A strengthened version of this fact was obtained by Kosniowski and Stong \cite{ks}: if $(M^m,T)$ is a non-bounding involution, which is equivalent to the fact that the
normal bundle of $F$ in $M^m$ is not a boundary (see \cite{cf}), then $m\leqslant\frac{5}{2}n$. In particular, if $F$ is non-bounding (which means that at least one $F^j$ is nonbounding), then $m\leqslant\frac{5}{2}n$.
The case $F=F^n$ was settled by C. Kosniowski and R. E. Stong \cite{ks}.
When $F=F^n\cup \{point\}$, recently Stong and Pergher \cite{ps} proved the following result for each natural number $n$, write
$n = 2^{p}q$, where $p\geq 0$ and $q$ is odd, and set
\[
\mathcal{M}(n)=
\begin{cases}
(2^{p+1}-1)q+p+1 = 2n+p-q+1,  & \textrm{if $p\leqslant q$,}\\
(2^{p+1}-2^{p-q})q+2^{p-q}(q+1)=2n+2^{p-q}, & \textrm{if $p>q$}.\\
\end{cases}
\]
Then, if $(M^m,T)$ is an involution whose fixed set has the form $F=F^n\cup \{point\}$, $m\leqslant\mathcal{M}(n)$. Further, there are involutions with $m=\mathcal{M}(n)$ fixing a point and some $F^n$, thus showing that the bounds $\mathcal{M}(n)$ cannot be improved.\\
The next case is $F=F^n\cup F^j$ with $0<j<n$; in a series of papers, Pergher et al. (See \cite{fp-topology}, \cite{bp}.) dealt with the subcases $j\leqslant 3$ and $j=n-1$ (See \cite{fp-dedicata}.).\\
Let $M^m$ be a closed smooth manifold equipped with a smooth involution having fixed point set of the form $F^n\cup F^4$, where $F^n$ and $F^4$ are submanifolds with dimensions $n$ and $4$, respectively, and where $4<n<m$ and $F^n\cup F^4$ does not bound. In this work, we establish the upper bound for $m$, for each $n$. 
Now, we state the main result of our paper.
\begin{theorem}\label{main theorem}
Suppose that $(M^m,T)$ is an involution having fixed set $F$ which does not bound and has the form  $F^n\cup F^4$. If the normal bundle over the component $F^4$ does not bound, then $m\leqslant\mathcal{M}(n-4)+8$.
\end{theorem} 

\section{Preliminaries}

Let $\eta$ be a general $k$-dimensional vector bundle over a closed smooth $n$-dimensional manifold $N$. Write $W(\eta)=1+w_1(\eta)+w_2(\eta)+ \cdots +w_k(\eta)\in H^{*}(N,\mathbb{Z}_2)$ for the Stiefel–Whitney class of $\eta$, and 
$W(N)=1+w_1(N)+w_2(N)+ \cdots +w_n(N)$ for the Stiefel–Whitney class of the tangent bundle of $N$. Then the fiber bundle cobordism class of $\eta$ is determined by the set of Whitney numbers (or characteristic numbers ) of $\eta$; such modulo $2$ numbers are obtained by evaluating $n$-dimensional $\mathbb{Z}_2$-cohomology classes of the form
\[
w_{i_1}(N)w_{i_2}(N)\ldots w_{i_r}(N)w_{j_1}(\eta)w_{j_2}(\eta)\ldots w_{j_s}(\eta)\in H^{n}(N,\mathbb{Z}_2)
\]
(that is, with $i_1+i_2+\cdots i_r+j_1+j_2+\cdots j_s=n$) on the fundamental homology class $[N]\in H_n(N,\mathbb{Z}_2)$.
Furthermore, we need to know information about the structure of the cohomology ring and the Stiefel–Whitney class of $\mathbb{R}P(\eta)$. If
$\lambda\mapsto\mathbb{R}P(\eta)$ is the Hopf line bundle, set $w_1(\lambda)=c\in H^1(\mathbb{R}P(\eta), \mathbb{Z}_2)$. From \cite{bh}, one has
\begin{multline*}
W(\mathbb{R}P(\eta))=(1+w_1(N)+w_2(N)+\cdots+w_n(N)) [(1+c)^k\\
 \hspace{3cm}    +(1+c)^{k-1}w_1(\eta)+\cdots+(1+c)w_{k-1}(\eta)+w_k(\eta)],\hfill 
\end{multline*}
where here we are suppressing bundle maps. Furthermore, this gives the relation
\[c^k+ c^{k-1}w_1(\eta) +c^{k-2}w_2(\eta) +\cdots+ w_k(\eta)=0. \]
The structure of the $\mathbb{Z}_2$-cohomology ring of $\mathbb{R}P(\eta)$ is then determined by the above relation and from the Leray–Hirsch Theorem (See \cite{Borel1953}.), which yields
that $H^*(\mathbb{R}P(\eta), \mathbb{Z}_2)$ is a free graded $H^*(N, \mathbb{Z}_2)$-module with basis $1, c, c^2,\ldots, c^{k-1}$.

 Let $(\eta\mapsto F^n) \cup (\mu\mapsto F^4)$ be the fixed data of $(M,T)$, write $W(F^4)=1+w_1+w_2+w_3+w_4$ for the  Stiefel–Whitney class of $F^4$ and $W(\mu)=1+v_1+v_2+v_3+v_4$  for the  Stiefel–Whitney class of $\mu$.
\section{The bound $m\leqslant\mathcal{M}(n-4)+8$}
\noindent
This section will be devoted to the proof of the part ``$m\leqslant\mathcal{M}(n-4)+8$" of Theorem \ref{main theorem}.
\begin{lemma}\label{main lemma}
If $m>\mathcal{M}(n-4)+8$, then $v_1^4=w_1^4$, $v_2^2+w_2^2=v_1^2w_1^2$, $v_1v_3=v_1^2w_2+w_1^2v_2$, $v_1^2v_2+v_1v_3=v_1^3w_1+v_1w_1^3$ and $v_4+w_4=v_1^2w_2+v_1w_3$.
\end{lemma}
Note that the unique $\beta$ satisfying relations of Lemma \ref{main lemma} is the stable cobordism class corresponding to the bundle $3\xi\mapsto\mathbb{R}P^4$.
Thus this lemma will reduce our task to the following:
\begin{theorem}\label{main theorem 2}
Let $(M^m,T)$ is an involution having fixed set $F$ of the form  $F=F^n\cup F^4$. If the normal bundle $\mu\mapsto F^4$ represents $\beta$, then $m\leqslant\mathcal{M}(n-4)+8$.	
\end{theorem}
The following basic fact from \cite{cf} will be needed for the proof of Lemma \ref{main lemma}: Let $\mathbb{R}P(\eta)$ and $\mathbb{R}P(\mu)$ denote the projective space bundles, and $\lambda\mapsto\mathbb{R}P(\eta)$ and  $\nu\mapsto\mathbb{R}P(\mu)$ denote the line bundles of the double covers $S(\eta)\mapsto\mathbb{R}P(\eta)$ and $S(\mu)\mapsto\mathbb{R}P(\mu)$ where $S(\;)$ meaning sphere bundles, then $\lambda\mapsto\mathbb{R}P(\eta)$ and  $\nu\mapsto\mathbb{R}P(\mu)$ are cobordant as elements of the bordism group $\mathcal{N}_{m-1}(BO(1))$. Then any class of dimension $m-1$, given by a product of the classes $w_i(\mathbb{R}P(\eta))$ and $w_1(\lambda)$, evaluated on the fundamental homology class $[\mathbb{R}P(\eta)]$, gives the same characteristic number as the one obtained by the corresponding product of the classes  $w_i(\mathbb{R}P(\mu))$ and $w_1(\nu)$, evaluated on $[\mathbb{R}P(\mu)]$. Set $k=m-n$, and write
\begin{eqnarray*}
W(F^n)=1+\theta_1+\cdots+\theta_n,\\
W(\eta)=1+u_1+\cdots+u_k \quad and \\
W(\lambda)=1+c
\end{eqnarray*}
for the Stiefel–Whitney classes of $F^n$, $\eta$ and $\lambda$, respectively.
It is well-known from \cite{bh} that 
\[W(\mathbb{R}P(\eta))=(1+\theta_1+\theta_2+\cdots+\theta_n)\{(1+c)^k+(1+c)^{k-1}u_1+\cdots+(1+c)u_{k-1}+u_k \}\] where here we are suppressing bundle maps.
First, for any integer $r$, Stong and Pergher \cite{ps} introduced the following variant of $W(\mathbb{R}P(\eta))$:
\[W[r]=\frac{W(\mathbb{R}P(\eta))}{(1+c)^{k-r}},\]
noting that each class $W[r]_j$
is still a polynomial in the classes $w_i(\mathbb{R}P(\eta))$ and $c$.
Further, these classes satisfy the following special properties
\[
W[r]_{2r}=\theta_{r}c_r+\text{terms with smaller $c$ powers,}
\]
\[
W[r]_{2r+1}=(\theta_{r+1}+u_{r+1})c_r+\text{terms with smaller $c$ powers.}
\]

For $n\geqslant 5$, write $n-4=2^{p}q$, where $p\geqslant 0$ and $q$ is odd, and suppose first that $p<q+1$. Consider the list of integers $r_1, r_2,\ldots, r_p$, where $r_i=2^p-2^{p-i}$, and take the class 
\[
X=W[2^p-1]_{2^{p+1}-1}^{q+1-p}\cdot W[r_1]_{2r_1}\cdot W[r_2]_{2r_2}\cdots W[r_p]_{2r_p}
\]

The dimension of $X$ is
\[
 (q+1-p)(2^{p+1}-1)+\mathlarger{\sum}_{i=1}^{p}(2^p-2^{p-i})=(2^{p+1}-1)q+p+1=\mathcal{M}(n-4)
\]
From the properties above listed, one has

\begin{multline*}
X=(\theta_{2^p}+u_{2^p})c^{2^p-1}+\textrm{terms with smaller $c$ powers})^{q+1-p}(\theta_{r_1}c^{r_1}+\textrm{terms with smaller $c$ powers})\cdots\\ 
 \cdot(\theta_{r_p}c^{r_p}+\textrm{terms with smaller $c$ powers})\\
=((\theta_{2^p}+u_{2^p})^{q+1-p}\cdot\theta_{r_1}\cdot\theta_{r_2}\cdots\theta_{r_p})c^{(q+1-p)(2^p-1)+\sum_{i=1}^{p}r_i}+\textrm{terms with smaller $c$ powers}
\end{multline*}
Note that $(q+1-p)2^p+\sum_{i=1}^{p}r_i=(q+1-p)2^p+p2^p-2^p+1=2^{p}q+1=n-3$\\
Now if $p\geqslant q+1$, take
\[
X=W[r_1]_{2r_1}\cdot W[r_2]_{2r_2}\cdots W[r_{q+1}]_{2r_{q+1}}
\]
where again $r_i=2^p-2^{p-i}$ for $1\leqslant i\leqslant q+1$.\\
The dimension of $X$ is\\
$2\mathlarger{\sum}_{i=1}^{q+1}(2^p-2^{p-i})=(q+1)2^p-2^p+2^{p-q-1}=2^{p}q+2^{p-q-1}=n-4+2^{p-q-1}\geqslant n-3$.\\
Thus for every $n\geqslant 5$, $X$ is a class of dimension $\mathcal{M}(n-4)$ which has the form
\[X=A_t\cdot c^{\mathcal{M}(n-4)-t}+ \text{terms with smaller $c$ powers},
\]
where $A_t$ is a cohomology class of dimension $t\geqslant n-3$ coming from the cohomology of $F^n$.\\
Our next crucial step to introduce the five special $8$-dimensional
cohomology classes $f_{\omega_i}(\lambda), i=1,2,3,4,5$, where $\lambda$ is a line bundle over a smooth closed $s$-dimensional manifold $B^s$. Using the splitting, we write 
\[
W(B^s)=(1+x_1)\cdot(1+x_2)\cdots(1+x_s) 
\]
and $W(\lambda)=1+c$. Consider the following symmetric polynomials in the variables $x_1,x_2,\ldots, x_s,c$, of degree $8$, and related to the partitions of $4$, $\omega_1=(1,1,1,1)$, $\omega_2=(2,1,1)$, $\omega_3=(2,2)$,
$\omega_4=(3,1)$, $\omega_5=(4)$:
\begin{gather*}
\begin{multlined}
f_{\omega_1}=\sum_{i<j<k<l}x_i(c+x_i)x_j(c+x_j)x_k(c+x_k)x_l(c+x_l),\\
f_{\omega_2}=\sum_{\stackrel{i<j}{i,j\neq k}
}x_i(c+x_i)x_j(c+x_j)x_k^2(c+x_k)^2,\hfill\\
f_{\omega_3}=\sum_{i<j}x_i^2(c+x_i)^2 x_j^2(c+x_j)^2,\hfill\\
f_{\omega_4}=\sum_{i\neq j}x_i(c+x_i) x_j^3(c+x_j)^3 \quad \text{and}\hfill\\
f_{\omega_5}=\sum_{i}x_i^4(c+x_i)^4.\hfill\\
\end{multlined}
\end{gather*}
Then $f_{\omega_1}$, $f_{\omega_2}$, $f_{\omega_3}$, $f_{\omega_4}$ and $f_{\omega_5}$ determine polynomials of dimension $8$ in the classes $w_i(B^s)$ and $w_1(\lambda)=c$. Specializing for $\lambda\mapsto\mathbb{R}P(\eta)$, we write
\begin{align*}
W(F^n)=(1+x_1)\cdot(1+x_2)\cdots(1+x_n)\quad \text{and}\\
W(\eta)=(1+y_1)\cdot(1+y_2)\cdots(1+y_k).\hfill\\
\end{align*}
Then \[
W(\mathbb{R}P(\eta))=(1+x_1)(1+x_2)\cdots(1+x_n)(1+c+y_1)(1+c+y_2)\cdots(1+c+y_k).
\]
It follows that 
\begin{multline*}
f_{\omega_1}(\lambda\mapsto\mathbb{R}P(\eta))=\Big(\sum_{i<j<k<l}x_i x_j x_k x_l+\sum_{i<j<k<l}y_i y_j y_k y_l+ \sum_{\stackrel{i,j,k,l}{i<j, k<l}}x_i x_j y_k y_l\hfill\\ 
+\sum_{\stackrel{i,j,k,l}{j<k<l}}x_i y_j y_k y_l
+\sum_{\stackrel{i,j,k,l}{i<j<k}}x_i x_j x_k y_l\Big)c^4+ \text{terms with smaller $c$ powers,} 
\end{multline*}

\begin{multline*}
f_{\omega_2}(\lambda\mapsto\mathbb{R}P(\eta))=\Big(\sum_{\stackrel{i,j,k}{i<j;\;i,j\neq k}}x_i x_j x_k^2 + \sum_{\stackrel{i,j,k}{i<j;\;i,j\neq k}}y_i y_j y_k^2+\sum_{\stackrel{j<k}{i\neq j,k}}x_i^2 y_j y_k + \sum_{\stackrel{i<j}{i\neq j,k}}x_i x_j y_k^2 \\ 
+\sum_{\stackrel{i,j,k}{j\neq k}}x_i y_j y_k^2+\sum_{\stackrel{i,j,k}{i\neq j}}x_i^2 x_j y_k \Big)c^4+ \text{terms with smaller $c$ powers,}
\end{multline*}
\begin{multline*}
	f_{\omega_3}(\lambda\mapsto\mathbb{R}P(\eta))=\Big(\sum_{i<j}x_i^2x_j^2 +\sum_{k<l}y_k^2y_l^2+\sum_{i<k}x_i^2y_k^2\Big)c^4+ \text{terms with smaller $c$ powers,}  
\end{multline*}
\begin{multline*}
f_{\omega_4}(\lambda\mapsto\mathbb{R}P(\eta))=\Big(\sum_{i\neq j}x_ix_j^3 + \sum_{i\neq j}y_iy_j^3+\sum_{i,j}x_i^3y_j + \sum_{i,j}x_iy_j^3\Big)c^4+ \text{terms with smaller $c$ powers,}   
\end{multline*}
\begin{multline*}
f_{\omega_5}(\lambda\mapsto\mathbb{R}P(\eta))=\Big(\sum_{i}x_i^4+\sum_{j}y_j^4\Big)c^4 +\text{terms with smaller $c$ powers.} \hfill
\end{multline*}
Therefore, every term of $f_{\omega_1}$, $f_{\omega_2}$, $f_{\omega_3}$, $f_{\omega_4}$ and $f_{\omega_5}$ has a factor of dimension at least $4$ from the cohomology of $F^n$. On the other hand, each term of our previous class $X$ has a factor of dimension at least $n-3$ from the cohomology of $F^n$,
which means that, for $i=1,2,3,4,5$, $f_{\omega_i}\cdot X$ is a class in $H^{\mathcal{M}(n-4)+8}(\mathbb{R}P(\eta), \mathbb{Z}_2)$ with each one of its terms having a factor of dimension at least $n+1$ from $F^n$. Thus $f_{\omega_i}(\lambda)\cdot X=0$. Since $m>\mathcal{M}(n-4)+8$, we can form the class
$f_{\omega_i}(\lambda)\cdot X\cdot c^{m-1-(\mathcal{M}(n-4)+8)}$, which yields the zero characteristic number, $f_{\omega_i}(\lambda)\cdot X\cdot c^{m-1-(\mathcal{M}(n-4)+8)}[\mathbb{R}P(\eta)]$.\\
Hence, the left side of our system of equations is zero, and thus the next task is to analyze the right side of it. Next we analyze the class associated to $\nu\mapsto\mathbb{R}P(\mu)$ which corresponds to $f_{\omega_i}(\lambda)\cdot X\cdot c^{m-1-(\mathcal{M}(n-4)+8)}$.
Setting $W(\nu)=1+d$, this class is\\
$f_{\omega_i}(\nu\mapsto\mathbb{R}P(\mu))\cdot Y\cdot d^{m-1-(\mathcal{M}(n-4)+8)}$, where $Y$ is obtained from $X$ by replacing each $W[r]_i$ by $W[n+r-4]_i$.
We write $W(F^4)=(1+x_1)(1+x_2)(1+x_3)(1+x_4)$ and\\ $W(\mu)=(1+y_1)(1+y_2)(1+y_3)(1+y_4)$ and denote by $\sigma_i$
the $i$-th elementary symmetric polynomial in the variables $x_1, x_2, x_3, x_4, y_1, y_2, y_3\;\text{and}\; y_4$; that is
\[(1+x_1)(1+x_2)(1+x_3)(1+x_4)(1+y_1)(1+y_2)(1+y_3)(1+y_4)=1+\sigma_1+\sigma_2+\sigma_3+\sigma_4+\sigma_5+\sigma_6+\sigma_7+\sigma_8. \]
This is the factored form of the Whitney sum $\tau\oplus\mu$, where $\tau$ is the tangent bundle over $F^4$. The Stiefel–Whitney class of $\mathbb{R}P(\mu)$
\begin{multline*}
W(\mathbb{R}P(\mu))=(1+w_1+w_2+w_3+w_4)\{(1+d)^{n+k-4}+(1+d)^{n+k-5}v_1
+(1+d)^{n+k-6}v_2\\
\hspace{3cm}+(1+d)^{n+k-7}v_3+(1+d)^{n+k-8}v_4\}\hfill
\end{multline*}
Rewriting,
\begin{multline*}
W(\mathbb{R}P(\mu))=(1+d)^{n+k-8}\{(1+w_1+w_2+w_3+w_4)\{(1+d)^4+(1+d)^3v_1
+(1+d)^2v_2\hfill\\
\hspace*{2cm}+(1+d)v_3+v_4\}\}\hfill\\
\hspace*{2cm}=(1+d)^{n-k+8}\cdot(1+x_1)\cdot(1+x_2)\cdot(1+x_3)\cdot(1+x_4)\cdot(1+d+y_1)\cdot(1+d+y_2)\\
\hspace*{2cm}\cdot(1+d+y_3)\cdot(1+d+y_4).\hfill
\end{multline*}
Since $d+d=0$, note that the part $(1+d)^{n-k+8}$ does not contribute to $f_{\omega_i}(\nu)$. Then, by a straightforward calculation, we obtain 
\begin{gather*}
\begin{multlined}
f_{\omega_1}(\nu)=\sigma_4d^4+\text{terms with smaller $d$ powers},\hfill\\
f_{\omega_2}(\nu)=(\sigma_1\sigma_3+\sigma_4)d^4+\text{terms with smaller $d$ powers},\hfill\\
f_{\omega_3}(\nu)=(\sigma_2^2+\sigma_1\sigma_3+\sigma_4)d^4+\text{terms with smaller $d$ powers},\hfill\\
f_{\omega_4}(\nu)=(\sigma_1^2\sigma_2+\sigma_2^2+\sigma_1\sigma_3+\sigma_4)d^4+\text{terms with smaller $d$ powers},\hfill\\
f_{\omega_5}(\nu)=(\sigma_1^4+\sigma_1^2\sigma_2+\sigma_2^2+\sigma_1\sigma_3+\sigma_4)d^4+\text{terms with smaller $d$ powers}.\hfill\\
\end{multlined}
\end{gather*}
Set $W(\tau\oplus\mu)=1+V_1+V_2+V_3+V_4$ and note that if a term (with dimension $8$) has a power of $d$ less than $4$, it necessarily has a factor of dimension greater than $4$ from the cohomology of $F^4$. Hence, we obtain
\begin{gather*}
\begin{multlined}
f_{\omega_1}(\nu)=V_4\hfill\\
f_{\omega_2}(\nu)=V_1V_3+V_4,\hfill\\
f_{\omega_3}(\nu)=V_2^2+V_1V_3+V_4,\hfill\\
f_{\omega_4}(\nu)=V_1^2V_2+V_2^2+V_1V_3+V_4,\hfill\\
f_{\omega_5}(\nu)=V_1^4+V_1^2V_2+V_2^2+V_1V_3+V_4.\hfill\\
\end{multlined}
\end{gather*}
Let $\mathcal{I}$ be the ideal of $H^*(\mathbb{R}P(\mu), \mathbb{Z}_2)$ generated by the classes coming from $F^4$ and with positive dimension, we obtain by dimensional reasons that $f_{\omega_i}(\nu)\cdot A=0$ for each $A\in\mathcal{I}$. Thus, in the computation of $Y$, we need to consider only that 
\[
W(\mathbb{R}P(\mu))\equiv(1+d)^{n+k-4}~\text{mod $\mathcal{I}$}
\]
and, for each integer $l$,
\[
W[l]\equiv (1+d)^l~\text{mod $\mathcal{I}$.}
\]
For $r_i=2^p-2^{p-i}$, $i=1,\ldots, p$, set $l_i=n+r_i-4=2^{p}q+2^p-2^i$. Then
\[
W[l_i]_{2r_i}\equiv\binom{2^{p}q+2^p-2^{p-i}}{2^{p+1}-2^{p-i+1}}d^{2r_i}~\text{mod $\mathcal{I}$.}
\]
Also if $r=2^p-1$, $l=n+r-4=2^pq+2^p-1$ and 
\[
W[l]_{2r+1}\equiv\binom{2^{p}q+2^p-1}{2^{p+1}-1}d^{2r+1} ~\text{mod $\mathcal{I}.$}
\]
The lesser term of the $2$-adic expansion of $2^pq+2^p$ is $2^{p+1}$. By Lucas' theorem, we conclude that the above binomial coefficients are nonzero modulo $2$. It follows that all classes $W[r]_i$ occurring in $Y$ satisfy $W[r]_i\equiv d^i$ \text{mod $\mathcal{I}$}, which implies that $Y\equiv d^{\mathcal{M}(n-4)}$ \text{mod $\mathcal{I}$}. Thus, if $b\in H^4(\mathbb{R}P(\mu), \mathbb{Z}_2)$ is a cohomology class coming from $H^4(F^4, \mathbb{Z}_2)$, since $H^*(\mathbb{R}P(\mu), \mathbb{Z}_2)$ is the free $H^*(F^4, \mathbb{Z}_2)$-module on $1,d,d^2,\ldots, d^{n+k-5}$, we have that 
\[
b\cdot d^4\cdot Y\cdot d^{m-1-(\mathcal{M}(n-4)+8)}[\mathbb{R}P(\mu)]=b\cdot d^{m-5}[\mathbb{R}P(\mu)]=b[F^4],
\]
and $b[F^4]=0$ if and only if $b=0$.\\
Thefrefore, our system of equations becomes the cohomological
system of equations
\begin{equation*}
\begin{cases} 
0=V_4\\
0=V_1V_3+V_4\\
0=V_2^2+V_1V_3+V_4\\
0=V_1^2V_2+V_2^2+V_1V_3+V_4\\
0=V_1^4+V_1^2V_2+V_2^2+V_1V_3+V_4.
\end{cases}
\end{equation*}
Thus we conclude that $V_4=0$, $V_1V_3=0$, $V_2^2=0$, $V_1^2V_2=0$ and $V_1^4=0$.
Since $V_1=v_1+w_1$, we get $v_1^4=w_1^4$, the first relation of the lemma. Now we have $0=V_2^2=(v_2+w_2+v_1w_1)^2=v_2^2+w_2^2+v_1^2w_1^2$, the second relation of the lemma. Next $0=V_1V_3=(v_1+w_1)(v_3+w_3+v_1w_2+v_2w_1)=v_1v_3+v_1w_3+v_1^2w_2+v_1v_2w_1+w_1v_3+w_1w_3+v_1w_1w_2+v_2w_1^2$.
Now let $U=1+u_1+u_2$ be the Wu class of $F^4$. Then \[W(F^4)=1+w_1+w_2+w_3+w_4=Sq(U)=1+u_1+u_1^2+u_2+Sq^1(u_2)+u_2^2, \] that is,
$w_1=u_1$, $w_2=u_1^2+u_2$, $w_3=Sq^1(u_2)$, $w_4=u_2^2$, which gives $u_2=w_1^2+w_2$. Since $Sq^1(w_1^2)=0$, by the Wu formula, $w_3=Sq^1(u_2)=Sq^1(w_1^2+w_2)=Sq^1(w_2)=w_1w_2$. 
If $x\in H^{4-k}(F^4, \mathbb{Z}_2)$, it is known that $Sq^k(x)=u_k x$.
Then we obtain $w_2^2=Sq^2(w_2)=u_2w_2=(w_1^2+w_2)w_2$ which implies $w_1^2w_2=0$. Also, by the Wu formula, we have  $w_1v_3=u_1v_3=Sq^1(v_3)=v_1v_3$ and $v_1v_2w_1=Sq^1(v_1v_2)=v_1Sq^1(v_2)+Sq^1(v_1)v_2=v_1(v_1v_2+v_3)+v_1^2v_2=v_1v_3$.
Then, from above relations, $V_1V_3=0$ reduces to $0=v_1^2w_2+v_1v_3+w_1^2v_2$. So, we have $v_1v_3=v_1^2w_2+w_1^2v_2$, the third relation of the lemma. Next $0=V_1^2V_2=(v_1^2+w_1^2)(v_2+w_2+v_1w_1)=v_1^2v_2+v_1^2w_2+v_1^3w_1+w_1^2v_2+w_1^2w_2+v_1w_1^3$. Since $v_1^2w_2+v_1v_3+w_1^2v_2=0$, we have the following reduction $v_1^2v_2+v_1v_3=v_1^3w_1+v_1w_1^3$, the fourth relation of the lemma.
Finally, $0=V_4=v_4+w_4+v_1w_3+v_2w_2+v_3w_1=0$. We know 
$v_3w_1=Sq^1(v_3)=v_1v_3$ and form the third relation of the lemma, $v_1^2w_2=w_1^2v_2+v_1v_3$. Now
$v_2^2=Sq^2(v_2)=u_2v_2=(w_1^2+w_2)v_2=w_1^2v_2+v_2w_2$ which gives $v_2w_2=v_2^2+w_1^2v_2$.
Hence, from above relations, we have the following reduction to $V_4=0$:
$v_4+w_4+v_1^2w_2+v_1w_3=0$. So, we obtain $v_4+w_4=v_1^2w_2+v_1w_3$, the fifth relation of the lemma. Hence Lemma \ref{main lemma} is proved.

Now we prove Theorem \ref{main theorem 2}. Let $\mathbb{R}P^n$ be the $n$-dimensional real projective space, and $\xi\mapsto\mathbb{R}P^n$ be the canonical line bundle 
and $j\epsilon_{\mathbb{R}}\mapsto X$ denote the $j$-dimensional trivial vector bundle over $X$. It is not hard to see that $\mu\mapsto F^4=3\xi\oplus (m-7)\epsilon_{\mathbb{R}}\mapsto\mathbb{R}P^4$ is a representative for $\beta$ and we want to show that $m\leqslant\mathcal{M}(n-4)+8$. We repeat the notations $\nu\mapsto\mathbb{R}P(\mu)$ and $W(\nu)=1+d$ for the standard line bundle over $\mathbb{R}P(\mu)$ and its characteristic class. Our strategy will consist in showing that, if $m>\mathcal{M}(n-4)+8$, then it is possible to find
polynomials in the characteristic classes so that the corresponding characteristic numbers are zero on $F^n$ and nonzero on $F^4$, thus giving the contradiction.
The following will be a subtle calculation based on the structure of the cohomology ring of $\mathbb{R}P(\mu)$. Let $\alpha\in H^1(\mathbb{R}P^4, \mathbb{Z}_2)$ be the generator. Since $H^*(\mathbb{R}P(\mu), \mathbb{Z}_2)$ is the free $H^*(F^4, \mathbb{Z}_2)$-module on $1,d,d^2,\ldots, d^{m-5}$, subject to the relation
$d^{m-4}=d^{m-5}\alpha+d^{m-6}\alpha^2+d^{m-7}\alpha^3$. From this relation, we obtain 
\begin{equation*}
\begin{aligned}
d^{m-1} &=d^{m-2}\alpha+d^{m-3}\alpha^2+d^{m-4}\alpha^3,\\ 
d^{m-2} &=d^{m-3}\alpha+d^{m-4}\alpha^2+d^{m-5}\alpha^3 \quad \text{and}\\ 
d^{m-3} & =d^{m-4}\alpha+d^{m-5}\alpha^2+d^{m-6}\alpha^3.
\end{aligned}
\end{equation*}
Therefore, we obtain 
$d^{m-2}\alpha=d^{m-3}\alpha^2+d^{m-4}\alpha^3+d^{m-5}\alpha^4$ and  $d^{m-4}\alpha^3=d^{m-5}\alpha^4$. Combining above relations, we obtain 
$d^{m-2}\alpha=d^{m-3}\alpha^2$. This implies $d^{m-1}=d^{m-4}\alpha^3$.
Multiplying by $\alpha^2$ to the relation $d^{m-3}=d^{m-4}\alpha+d^{m-5}\alpha^2+d^{m-6}\alpha^3$, we obtain $d^{m-3}\alpha^2=d^{m-4}\alpha^3+d^{m-5}\alpha^4$. 
Now $d^{m-4}\alpha^3=d^{m-5}\alpha^4$ yields $d^{m-3}\alpha^2=0$.
Therefore, we have $d^{m-2}\alpha=d^{m-3}\alpha^2=0$.
Combining these relations, we obtain $d^{m-1}=d^{m-4}\alpha^3=d^{m-5}\alpha^4$,
which is the (top-dimensional) generator of $H^{m-1}(\mathbb{R}P(\mu), \mathbb{Z}_2)$.\\
First consider $n$ odd. In this case, we will obtain a stronger result, noting that
$\mathcal{M}(n-4)+8=n+5$.


\begin{lemma}
If $(M^m, T)$ is an involution fixing  $F=F^n\cup F^4$, where $n$ is odd and 
$\mu\mapsto F^4=3\xi\oplus (m-7)\epsilon_{\mathbb{R}}\mapsto\mathbb{R}P^4$, then $m\leqslant n+3$ (hence $m=n+3$).
\end{lemma}

\begin{proof}
On $F^n$ one has 
$W[0]=(1+\theta_1+\theta_2+\cdots+\theta_n)\bigl\{1+\frac{u_1}{1+c}+\cdots+\frac{u_k}{(1+c)^k}\bigr\}$.
If $m\geq n+3$, one can form the class $W[0]_1^{n+3}c^{m-1-(n+3)}$ of dimension $m-1$. Since $W[0]_1^{n+3}=(\theta_1+u_1)^{n+3}$ comes from $F^n$,
this gives a zero characteristic number. The class over $F^4$ corresponding to $W[0]$ is $W[n-4]$. Now from 
\[
W(\mathbb{R}P(\mu))=(1+\alpha+\alpha^4)\bigl\{(1+d)^{m-4}+(1+d)^{m-5}\alpha+(1+d)^{m-6}\alpha^2+(1+d)^{m-7}\alpha^3  \bigr\} 
\]
we have 
\[
W[n-4]=(1+\alpha+\alpha^4)\bigl\{(1+d)^{n-4}+(1+d)^{n-5}\alpha+(1+d)^{n-6}\alpha^2+(1+d)^{n-7}\alpha^3  \bigr\}. 
\]
Since $n$ is odd, one has
\[
W[n-4]_1=\binom{n-4}{1}d+\alpha+\alpha=d.
\]
Thus, we have the nonzero characteristic number
\[
W[n-4]_1^{n+3}d^{m-1-(n+3)}[\mathbb{R}P(\mu)]=d^{m-1}[\mathbb{R}P(\mu)].
\]
\end{proof}
\noindent
Now we consider $n$ even, which implies in particular that $n\geqslant 6$.
Write $n-4=2^{p}q$, where $p,q\geqslant 1$. Over $F^n$, we take the same class $X$ considered before; that is $X\in
H^{\mathcal{M}(n-4)}(\mathbb{R}P(\eta), \mathbb{Z}_2)$ and each term of $X$
has a factor of dimension $n-3$ from the cohomology of $F^n$. Note that, on $F^n$, $W[0]_2=\theta_2+\theta_1u_1+u_1c+u_2$. Hence every term of 
$W[0]_2^4=\theta_2^4+\theta_{1}^{4}u_{1}^{4}+u_{1}^{4}c^{4}+u_{2}^4$ has a factor of dimension at least $4$ from $F^n$.
If $m>\mathcal{M}(n-4)+8$, we obtain the zero characteristic number
\[
X\cdot W[0]_2^4\cdot c^{m-1-(\mathcal{M}(n-4)+8)}[\mathbb{R}P(\eta)].
\]
So, the next and final task will be to show that, on $F^4$ the corresponding characteristic number
\[
Y\cdot W[n-4]_2^4\cdot d^{m-1-(\mathcal{M}(n-4)+8)}[\mathbb{R}P(\mu)]
\]
is nonzero.
Note that a general element of $H^t(\mathbb{R}P(\mu), \mathbb{Z}_2)$ is of the form $a_0 d^t+a_1\alpha d^{t-1}+a_2\alpha^2 d^{t-2}+a_3\alpha^3 d^{t-3}+a_4\alpha^4 d^{t-4}$. In particular, for the top-dimensional generator of $H^{m-1}(\mathbb{R}P(\mu), \mathbb{Z}_2)$, the number of 1's in
$\{a_0, a_1, a_2, a_3, a_4 \}$ is 1, 3 or 5.
From 
\[
W(\mathbb{R}P(\mu))=(1+\alpha+\alpha^4)\bigl\{(1+d)^{m-4}+(1+d)^{m-5}\alpha+(1+d)^{m-6}\alpha^2+(1+d)^{m-7}\alpha^3  \bigr\} 
\]
we obtain
\[
W[l]=(1+\alpha+\alpha^4)\bigl\{(1+d)^{l}+(1+d)^{l-1}\alpha+(1+d)^{l-2}\alpha^2+(1+d)^{l-3}\alpha^3  \bigr\} 
\]
and 
\begin{multline*}
W[l]_t=\binom{l}{t}d^t+\biggl\{\binom{l-1}{t-1}+\binom{l}{t-1} \biggr\}\alpha d^{t-1}+\biggl\{\binom{l-2}{t-2}+\binom{l-1}{t-2} \biggr\}\alpha^2 d^{t-2}+\hfill\\
+\biggl\{\binom{l-3}{t-3}+\binom{l-2}{t-3} \biggr\}\alpha^3 d^{t-3}+\biggl\{\binom{l-3}{t-4}+\binom{l}{t-4} \biggr\}\alpha^4 d^{t-4}.
\end{multline*}
To compute $Y$, we now write $r_i=2^p-2^i$, $i=0, 1,\ldots, p-1$, and set as before\\ $l_i=n+r_i-4=2^{p}q+2^p-2^i$.\\
Then
\begin{multline*}
W[l]_{2r_i}=\binom{2^{p}q+2^p-2^i}{2^{p+1}-2^{i+1}}d^{2r_i}+\biggl\{\binom{2^{p}q+2^p-2^i-1}{2^{p+1}-2^{i+1}-1}+\binom{2^{p}q+2^p-2^i}{2^{p+1}-2^{i+1}-1}\biggr\}\alpha d^{2r_i-1}\\
+\biggl\{\binom{2^{p}q+2^p-2^i-2}{2^{p+1}-2^{i+1}-2}+\binom{2^{p}q+2^p-2^i-1}{2^{p+1}-2^{i+1}-2}\biggr\}\alpha^2 d^{2r_i-2}\\ +\biggl\{\binom{2^{p}q+2^p-2^i-3}{2^{p+1}-2^{i+1}-3}+\binom{2^{p}q+2^p-2^i-2}{2^{p+1}-2^{i+1}-3}\biggr\}\alpha^3 d^{2r_i-3}\\
+\biggl\{\binom{2^{p}q+2^p-2^i-3}{2^{p+1}-2^{i+1}-4}+\binom{2^{p}q+2^p-2^i}{2^{p+1}-2^{i+1}-4}\biggr\}\alpha^4 d^{2r_i-4}.\\
\end{multline*}
By Lucas' theorem we have the following values for the above binomial coefficients:

\begin{enumerate}[(i)]

\item $\binom{2^{p}q+2^p-2^i}{2^{p+1}-2^{i+1}}\equiv\text{1 mod 2}$ \\
\item $\binom{2^{p}q+2^p-2^i-1}{2^{p+1}-2^{i+1}-1}\equiv\text{0 mod 2}$\\
\item $\binom{2^{p}q+2^p-2^i}{2^{p+1}-2^{i+1}-1}\equiv
\begin{cases}
	\text{1 mod 2},& \text{if $i=0$},\\
	\text{0 mod 2},& \text{if $i\geqslant 1$},
\end{cases}	
$\\
\item $\binom{2^{p}q+2^p-2^i-2}{2^{p+1}-2^{i+1}-2}\equiv
\begin{cases}
	\text{1 mod 2},& \text{if $i=0$},\\
	\text{0 mod 2},& \text{if $i\geqslant 1$},
\end{cases}	
$\\
\item $\binom{2^{p}q+2^p-2^i-1}{2^{p+1}-2^{i+1}-2}\equiv
\begin{cases}
	\text{1 mod 2},& \text{if $i=0$},\\
	\text{0 mod 2},& \text{if $i\geqslant 1$},
\end{cases}	
$\\
\item $\binom{2^{p}q+2^p-2^i-3}{2^{p+1}-2^{i+1}-3}\equiv
\begin{cases}
	\text{1 mod 2},& \text{if $i=1$},\\
	\text{0 mod 2},& \text{if $i=0$ or $i\geqslant 2$},
\end{cases}	
$\\
\item $\binom{2^{p}q+2^p-2^i-2}{2^{p+1}-2^{i+1}-3}\equiv\text{0 mod 2}$
\item $\binom{2^{p}q+2^p-2^i-3}{2^{p+1}-2^{i+1}-4}\equiv
\begin{cases}
	\text{1 mod 2},& \text{if $i=1$},\\
	\text{0 mod 2},& \text{if $i=0$ or $i\geqslant 2$},
\end{cases}	
$\\
\item $\binom{2^{p}q+2^p-2^i}{2^{p+1}-2^{i+1}-4}\equiv
\begin{cases}
	\text{1 mod 2},& \text{if $i=0, 1, 2$},\\
	\text{0 mod 2},& \text{if $i\geqslant 3$},
\end{cases}	
$\\
\end{enumerate}
From above, it follows that
\[
W[l]_{2r_i}=
\begin{cases}
d^{2r_i}+\alpha d^{2r_i-1}+\alpha^4 d^{2r_i-4},&  \text{if $i=0$},\\
d^{2r_i}+\alpha^3 d^{2r_i-3},& \text{if $i=1$},\\
d^{2r_i}+\alpha^4 d^{2r_i-4},& \text{if $i=2$},\\
d^{2r_i},& \text{if $i\geqslant 3$}.
\end{cases}
\]

For $r=2^p-1$, $l=n+r-4=2^pq+2^p-1$ and 
\begin{multline*}
W[l]_{2r+1}=\binom{2^{p}q+2^p-1}{2^{p+1}-1}d^{2r+1}+\biggl\{\binom{2^{p}q+2^p-2}{2^{p+1}-2}+\binom{2^{p}q+2^p-1}{2^{p+1}-2}\biggr\}\alpha d^{2r}\hfill\\
+\biggl\{\binom{2^{p}q+2^p-3}{2^{p+1}-3}+\binom{2^{p}q+2^p-1}{2^{p+1}-3}\biggr\}\alpha^2d^{2r-1}+\biggl\{\binom{2^{p}q+2^p-4}{2^{p+1}-4}+\binom{2^{p}q+2^p-3}{2^{p+1}-4}\biggr\}\alpha^3 d^{2r-2}\hfill\\
+\biggl\{\binom{2^{p}q+2^p-4}{2^{p+1}-5}+\binom{2^{p}q+2^p-1}{2^{p+1}-5}\biggr\}\alpha^4d^{2r-3}.\hfill
\end{multline*}
In the above expression, the unique binomial coefficient which is zero is
$\binom{2^{p}q+2^p-4}{2^{p+1}-5}$. Hence
$W[l]_{2r+1}=d^{2r+1}+\alpha^4d^{2r-3}$.
With these $l_i$'s and $l$, and for $p\leqslant q+1$ we obtain that
\begin{align*} 
Y=& (W[l]_{2r+1})^{q+1-p}\cdot\prod_{i=0}^{p-1} W[l]_{2r_i}\\
=& (d^{2r+1}+\alpha^4 d^{2r-3})^{q+1-p}\cdot (d^{2r_0}+\alpha d^{2r_0-1}+\alpha^4 d^{2r_0-4})\cdot(d^{2r_1}+\alpha^3 d^{2r_1-3})\cdot(d^{2r_2}+\alpha^4 d^{2r_2-4})\\
& \cdot d^{2(r_3+\cdots+r_{p-1})}.
\end{align*}
But 
\begin{align*}
(d^t+\alpha^4 d^{t-3})^s=& \sum_{i=0}^{s}\binom{s}{i}(d^t)^{s-i}(\alpha^4 d^{t-3})^i\\
=& d^{ts}+\binom{s}{1}d^{t(s-1)}\alpha^4 d^{t-3}\\
=& d^{ts}+s\alpha^4 d^{ts-3}.
\end{align*}
and $q+1-p\equiv p\text{ mod 2}$.\\
Therefore, we obtain
\[
Y=
\begin{cases}
	d^{\mathcal{M}(n-4)}+\alpha d^{\mathcal{M}(n-4)-1}+\alpha^3 d^{\mathcal{M}(n-4)-3}+\alpha^4 d^{\mathcal{M}(n-4)-4},& \text{if $p$ is even},\\
	d^{\mathcal{M}(n-4)}+\alpha d^{\mathcal{M}(n-4)-1}+\alpha^3 d^{\mathcal{M}(n-4)-3}, & \text{if $p$ is odd}.
\end{cases}
\]

For $p>q+1$, we have
\[
Y=\prod_{i=p-(q+1)}^{p-1} W[l]_{2r_i}=
\begin{cases}
	d^{\mathcal{M}(n-4)}+\alpha^3 d^{\mathcal{M}(n-4)-3}+\alpha^4 d^{\mathcal{M}(n-4)},& \text{if $p-(q+1)=1$},\\
	d^{\mathcal{M}(n-4)}+\alpha^4 d^{\mathcal{M}(n-4)},& \text{if $p-(q+1)=2$},\\
    d^{\mathcal{M}(n-4)},& \text{if $p-(q+1)>2$}.
\end{cases}
\]
Our final step is the calculation of $W[n-4]_2^4$ on $F^4$. We have
\[
W[n-4]=(1+\alpha+\alpha^4)\bigl\{(1+d)^{n-4}+(1+d)^{n-5}\alpha+(1+d)^{n-6}\alpha^2+(1+d)^{n-7}\alpha^3  \bigr\}. 
\]
and
\begin{align*} 
W[n-4]_{2}^{4}= &
\biggl(\binom{n-4}{2}d^2+\biggl(\binom{n-4}{1}+\binom{n-5}{1} \biggr)\alpha d \biggr)^4\\
=& \binom{2^pq}{2}d^8+\alpha^4d^4=
\begin{cases}
d^8+\alpha^4d^4,& \text{if $p=1$},\\
\alpha^4d^4,& \text{if $p>1$}.	
\end{cases}
\end{align*}
Since $Y$ has the form $d^t$, $d^t+\alpha d^{t-1}+\alpha^3 d^{t-3}$,
$d^t+\alpha^3 d^{t-3}+\alpha^4 d^{t-4}$, $d^t+\alpha^4 d^{t-4}$,
$d^t+\alpha d^{t-1}+\alpha^3 d^{t-3}+\alpha^4 d^{t-4}$,
for $p>1$, we obtain $Y\cdot W[n-4]_{2}^{4}=\alpha^4 d^{\mathcal{M}(n-4)+4}$.\\
If $p=1$, we have $d^{\mathcal{M}(n-4)}+\alpha d^{\mathcal{M}(n-4)-1}+\alpha^3 d^{\mathcal{M}(n-4)-3}$ and
\begin{align*}
Y\cdot W[n-4]_{2}^{4}=&
(d^{\mathcal{M}(n-4)}+\alpha d^{\mathcal{M}(n-4)-1}+\alpha^3      d^{\mathcal{M}(n-4)-3})\cdot (d^8+\alpha^4d^4)\\
=& d^{\mathcal{M}(n-4)+8}+\alpha d^{\mathcal{M}(n-4)+7}+\alpha^3 d^{\mathcal{M}(n-4)+5}+\alpha^4 d^{\mathcal{M}(n-4)+4}.
\end{align*}
In any case, $Y\cdot W[n-4]_2^4\cdot d^{m-1-(\mathcal{M}(n-4)+8)}[\mathbb{R}P(\mu)]$
is a nonzero characteristic number on $F^n$, which ends the proof.

\end{document}